\documentclass[sort&compress, AMA, STIX1COL]{WileyNJD-v2}

\articletype{Research article}%

\received{}
\revised{}
\accepted{}

\usepackage{mathrsfs,amsthm}
\usepackage{tikz}
\usepackage{tikz-3dplot}
\usetikzlibrary{calc,hobby}
\usetikzlibrary{decorations.pathmorphing,patterns,decorations,shapes,arrows, intersections,matrix,fit,calc,trees,positioning,arrows,chains, shapes.geometric,shapes,angles,quotes,fit,math}
\usetikzlibrary{arrows,positioning}
\usepackage{float}

\def\R#1{\mathbb{R}^{#1}}

\begin{document}

\title{Dynamical behaviour of SIR model with coinfection: the case of finite carrying capacity\protect\thanks{The paper is supported by the Swedish Research Council, grant \#.}}

\author[1]{Samia Ghersheen}

\author[1]{Vladimir Kozlov}

\author[1]{Vladimir G. Tkachev*}

\author[2]{Uno Wennergren}

\authormark{Samia Ghersheen, Vladimir Kozlov, Vladimir Tkachev  Uno Wennergren}

\address[1]{\orgdiv{Department of Mathematics}, \orgname{Link\"oping University}, \orgaddress{\state{Link\"oping}, \country{Sweden}}}

\address[2]{\orgdiv{Department of Physics, Chemistry, and Biology}, \orgname{Link\"oping University}, \orgaddress{\state{Link\"oping}, \country{Sweden}}}

\corres{*\email{vladimir.tkatjev@liu.se}}


\abstract[Summary]{ Multiple viruses are widely studied because of their negative effect on the health of host as well as on whole population. The dynamics of coinfection is important in this case.
 We formulated a SIR model that describes the coinfection of the two viral strains in a single host population with an addition of limited growth of susceptible in terms of carrying capacity. The model describes four classes of a population: susceptible, infected by first virus, infected by second virus, infected by both viruses and completely immune class. We proved that for any set of parameter values there exist a globally stable equilibrium point. This guarantees that the disease always persists in the population with a deeper connection between the intensity of infection and carrying capacity of population. Increase in resources in terms of carrying capacity promotes the risk of infection which may lead to destabilization of the population.}

\jnlcitation{\cname{%
\author{Ghersheen S.},
\author{V. Kozlov},
\author{V. Tkachev},
and
\author{U. Wennergren}} (\cyear{2018}),
\ctitle{Dynamical behaviour of SIR model with coinfection}, \cjournal{Math Meth Appl Sci.}, \cvol{}.}

\maketitle


\section{Introduction}
Coinfection with multiple strains in a single host is very common. Viral diseases such as AIDS/ HIV, Dengue fever, Hepatitis B and C are the great threats to human lives. Multiple strains of these viruses made the disease more sever and complicated to control. Sometimes coinfection may occur with multiple disease in one host such as HIV and Hepatitis B \cite{chun,Kang}, HIV and Hepatitis C \cite{Gupta}, Malaria and HIV \cite{Alemu}, DENV and ZIKV \cite{Chaves}, ZIKV and CHIKV \cite{Zambrano}.

Mathematical modelling of infectious diseases is an efficient tool for studying the dynamics of various virulent diseases which benefits to develop the appropriate strategies to control possible outbreaks of diseases. One of the most significant aspect of studying multi-strain epidemic models is to identify those conditions which lead to the coexistence of different strains.  The dynamics of coinfection is important in this case, because in case of co-infection treatment against one strain may agitate the other \cite{Gao}.

Many mathematical studies exist on interaction of multiple strains such as dengue virus \cite{Ferguson,Kawaguchi}, Influenza \cite{Sharp1}, human papilloma virus \cite{Chaturvedi} and multiple disease such as HIV/malaria \cite{Mukandavire}, HIV/pneumonia \cite{Nthiiri,Abu1},  Malaria/Cholera \cite{Okosun}.  Allen et al. \cite{Allen} studied a SI model with density dependent mortality and coinfection in a single host where one strain is vertically and the other is horizontally transmitted. The model has application on hantavirus and arenavirus. An ODEs model of co-infection was designed by Zhang et al \cite{Zhang} to study two parasite strains on two different hosts to know the sustainability and proliferation of these strains in response to variability in mode of action of parasites and its host types.  Bichara et al \cite{ Bichara} proposed  SIS, SIR and MSIR models with variable population, and n different pathogen strains to study that under generic conditions a competitive exclusion principle holds. A two disease model was also used by Martcheva and Pilyugin \cite{Martcheva} to study dynamics of dual infection by considering time of infection of primary disease.

Castillo et al \cite{ Castillo} analysed a SIS model on sexual transmitted disease by two hostile strains. Females with different susceptibility level to any of the virulent strain were separated into two groups. Stability analysis was performed to identify conditions for the co-existence and competitive exclusion of the two strains.  Gao et al. \cite{Gao} studies a SIS model with dual infection. Simultaneous transmission of infection and no immunity has been considered. The study revealed that the coexistence of multiple agents caused co-infection and made the disease dynamics more complicated. It was observed that coexistence of two disease can only occur in the presence of coinfection. In above models they considered that the number of births per unit time is constant.

In \cite{Sharp2} Sharp et al proposed a model for chronic wasting disease with density dependence to study the effect of density dependence and time delay on wildlife population and observed that more frequent  outbreaks of disease are caused by increased carrying capacity which leads to the disruption of a deer population. In contrast to the previous studies, we formulate a SIR model with coinfection and limited growth of susceptible population to study the effects of carrying capacity on disease dynamics. We  also carried out global stability analysis using a generalized Volterra function for each stable point to study the complete dynamics of disease. The model was formulated and some of our results were recently announced in \cite{SKTW18a}. We analyse the model with the possibility of transmission of two strains simultaneously. However, contrary to [11], to diminish the complexity of model and to study the global behaviour of the system, the reduction of the system is needed to some sense. So we assume that there is no interaction between single strains, since the co infected class is always the largest class. Our model also includes the fact that coinfection can occur as result of interaction between co infected class and single infected class and co infected class and susceptible class.
We analyse a SIR model with no cross immunity. In Sections \ref{equipoint} and \ref{sec8} we characterize all stable equilibrium points and give the results regarding global stability of all equilibrium points. In section \ref{H.C.C} we analyse the effect of carrying capacity on disease dynamics.

\section{Formulation of the model}
We consider a SIR model with the recovery of each class and assume that infected and recovered populations can not reproduce. A susceptible individual can be infected with both stains as a result of contact with co infected person. The disease induced death rate is ignored. We also assume that the co-infection can occur as a result of contact with the co-infected class. This process is illustrated in Fig.~\ref{flowdiag}.
\begin{figure}[h]
\begin{center}
\begin{tikzpicture}[every node/.style={ minimum height={1cm},minimum width={1cm},thick,align=center}]
\node[draw] (S) {S};
\node[draw, above right=of S] (I1) {$I_1$};
\node[draw, below=of I1] (I12) {$I_{12}$};
\node[draw, below=of I12] (I2) {$I_2$};
\node[draw, right= of I12] (Rec) {$R$};
\draw[->] (S) -- (I1); 
\draw[->] (S) -- (I12) ;
\draw[->,] (S) -- (I2);
\draw[->,] (I2) -- (I12);
\draw[->,] (I1) -- (I12);
\draw[->] (I1) -- (Rec);
\draw[->] (I2) -- (Rec);
\draw[->,] (I12) -- (Rec);
\node[left =of I1] at (2.5,1.3) {$\alpha_1$};
\node[left =of I2] at (2.5,-1.3) {$\alpha_2$};
\node[left =of I12]at (2.5,0.25) {$\alpha_3$};
\node[left =of I1] at (3.3,0.8) {$\eta_1$};
\node[left =of I1] at (3.3,-0.8) {$\eta_2$};
\node[left =of I1] at (4.5,0.25) {$\sigma_3$};
\node[left =of I1] at (4.5,1.3) {$\sigma_1$};
\node[left =of I1] at (4.5,-1.3) {$\sigma_2$};
\end{tikzpicture}
\end{center}
\vspace*{-0.5cm}
\caption{ Flow diagram for two strains coinfection model.}
\label{flowdiag}
\end{figure}
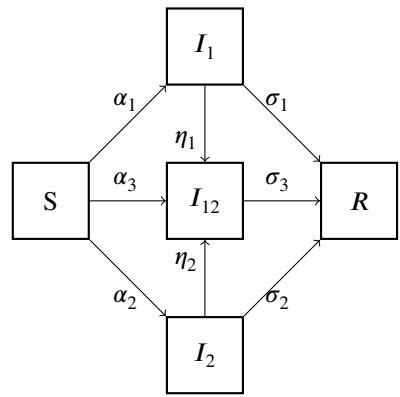

\noindent
The corresponding SIR model is then described by the ODE system as follows:
\begin{equation}\label{model}
\begin{split}
S'\,\, &=(b(1-\frac{S}{K})-\alpha_1I_1-\alpha_2I_2-\alpha_3I_{12}-\mu_0)S,\\
I_1'\,\, &=(\alpha_1S - \eta_1I_{12} - \mu_1)I_1,\\
I_2'\,\, &=(\alpha_2S - \eta_2I_{12}- \mu_2)I_2,\\
I_{12}' &=(\alpha_3S+ \eta_1I_1+\eta_2I_2-\mu_3)I_{12}, \\
R'\,\, &=\rho_1 I_1+\rho_2I_2+\rho_3 I_{12}-\mu_4' R.
\end{split}
\end{equation}
Here and in what follows we use the following notation:
\begin{itemize}
\item
$S$ represents the susceptible class,
\item
$I_1$ and $I_2$ represent infected classes from strain 1 and strain 2 respectively,
\item
$I_{12} $ represents co-infected class,
\item
$R$ represents the recovered class,
\item
$b$ is  the birthrate in the population,
\item
$K $ is  a carrying capacity,
\item
$\rho_i$ is the recovery rate from infected class $i$,
\item
$\mu_i $ is the reduced death rate of  class $i$,
\item
$\alpha_i$ is the transmission rate of strain $i$ (including the case of coinfection),
\item
$\eta_i$ is rate at which infected from one strain getting  infection from co-infected class $i$.
\end{itemize}

Let us make some natural comments about the present model.  First we suppose (and it also follows from \eqref{model})  that there is  no interaction between strain $1$ and strain $2$. Also, note that it is reasonable to assume that the death rate  of the susceptible class is less or equal than the corresponding reproduction rate because otherwise population will die out quickly. Therefore we assume always that
\begin{equation}\label{birthrel}
b-\mu_0>0.
\end{equation}

Furthermore, the system is considered under the natural initial conditions
\begin{equation}\label{initdata}
 S(0)>0,\quad I_1(0)>0, \quad I_2(0)>0,\quad  I_{12}(0)>0.
 \end{equation}
Indeed, it follows from the general theory of \eqref{model} that 1) any integral curve with \eqref{initdata} is staying in the non negative cone for all $t\ge0$, and, moreover, 2) if $S(0)=0$ or  $I_\alpha(0)=0$ for some index $\alpha$ then the corresponding coordinate will vanish for all $t\ge 0$.

Finally, note also that since the variable $R$ is not presented in the first four equations, we may consider only the first four equations of system \eqref{model}. Then $R(t)$ can be easily found by integrating the last (linear in $R$) equation in \eqref{model}.

To make a rigorous mathematical analysis of \eqref{model} it is convenient to keep the following unifying notation:
$$
S=Y_0,\quad
I_1=Y_{1}, \quad
I_{2}=Y_{2},\quad
I_{12}=Y_{3}.
$$
Then the first four equations of \eqref{model} can be rewritten in a compact Lotka-Volterra type form:
 \begin{equation}\label{gen}
 \begin{split}
 \frac{dY_k}{dt}&=F_{k}(Y)\cdot Y_k,\qquad k=0,1,2,3,
 \end{split}
 \end{equation}
  where we denote
   \begin{equation}\label{gen10}
   F(Y)=-q+AY,
   \end{equation}
  with
 \begin{equation}\label{gen11}
 F(Y)=
 \left(
 \begin{array}{c}
   F_0(Y)\\
  F_1(Y)\\
   F_2(Y)\\
   F_3(Y)
 \end{array}
 \right)
,\quad
q=
 \left(
 \begin{array}{c}
   -b+\mu_0 \\
   \mu_1\\
   \mu_2 \\
   \mu_3
 \end{array}
 \right)
,\quad
A=
  \left(
 \begin{array}{cccc}
   -\frac{b}{K}&-\alpha_1&-\alpha_2&-\alpha_3 \\
   \alpha_1&0&0&-\eta_1\\
   \alpha_2&0&0&-\eta_2 \\
   \alpha_3&\eta_1&\eta_2&0
 \end{array}
 \right)
 ,\quad
 Y=  \left(
 \begin{array}{c}
   Y_0 \\
   Y_1\\
   Y_2 \\
   Y_3
 \end{array}
 \right)
 \end{equation}
 A point $Y=(Y_0,Y_1,Y_2,Y_3)$ is called an equilibrium point of \eqref{gen} if
 \begin{equation}\label{FY0}
Y_iF_i(Y)=0, \qquad 0\le i\le 3.
\end{equation}

\smallskip
The following ratios play an essential role in our analysis:
 $$
 \sigma_i:=\frac{\mu_i}{\alpha_i}, \qquad 1\le i\le 3.
 $$
We shall always assume that the strains 1 and 2 \textit{are  different} in the sense $ \sigma_1 \neq \sigma_2$. Indeed, if $\sigma_1=\sigma_2$, it follows from the second and the third equations in \eqref{model} that the behaviour of the system lose the structural stability (i.e. the qualitative picture drastically depends on  small perturbations of the system parameters, in our case on the relations between $\alpha_i$, $\mu_i$ and $\eta_i$).

Then by change of the indices (if needed) we may assume that
  \begin{equation}\label{assum12}
\sigma_1 < \sigma_2.
\end{equation}
In other words, \eqref{assum12} means that strain 1 is more aggressive than strain 2.
Furthermore, it is natural to assume that the transmission rate of coinfection is always less than the transmission rates of the viruses 1 and 2, while the death rates $\mu_i$ are almost the same for different classes (as population groups). This  makes it natural to assume the following hypotheses:
  \begin{equation}\label{assum}
  \sigma_1< \sigma_2 < \sigma_3.
  \end{equation}

The vector of fundamental parameters
\begin{equation}\label{pparam}
p=(b,K,\mu_i,\alpha_j,\eta_k)\in \mathrm{int}(\R{11}_+),\quad \text{ where }0\le i\le 3, \,\,1\le j\le 3, \,\,1\le k\le 2,
\end{equation}
is said to be \textit{admissible} if \eqref{assum} holds.

A fundamentally important parameter for our study is the \textit{modified carrying capacity} defined by
\begin{equation}\label{capac}
S_2:=K(1-\frac{\mu_0}{b})>0.
\end{equation}
Note that the modified carrying capacity is always less than the carrying capacity. It expresses the (susceptible) population size in absence of any infection. More precisely, it follows from \eqref{model} that
$$
E_2:=(S_2,0,0,0)
$$
is an equilibrium point. Then $E_2$ represents the `healthy' state\footnote{To explain the natation: we denote by $E_1=\mathbf{0}$ the trivial equilibrium point and by $E_2$ the first nontrivial equilibrium state, see also \eqref{Ediag} below}, i.e. the equilibrium state with no infection and coinfection.

\section{Equilibrium points}\label{equipoint}

Below we  use the standard vector order relation:  given $x,y\in \mathbb{R}^n$,
\begin{itemize}
\item
$x\le y$ if $x_i\le y_i$ for all $1\le i\le n$,
\item
$x< y$  if $x\le y$ and $x\ne y$, and
\item $x\ll y$ if $x_i< y_i$ for all
$i$.
\end{itemize}
Then $\R{n}_{+}$ denotes the nonnegative cone $ \{x\in\R{n}:x\ge 0\}$
and for $a\le b$, $a,b\in\R{n}$,  $[a,b]=\{x\in \R{n}:a\le x\le b\}$
is the closed box with vertices at $a$ and $b$. By $\mathbf{0}$ we denote the origin in $\R{n}$.

By $\mathcal{E}(p)$ we denote the set of the equilibrium points of \eqref{gen} with nonnegative coordinates, i.e. those $Y^*=(Y_0^*,Y_1^*,Y_2^*,Y_3^*)\ge0$  satisfying
\begin{equation}\label{FY}
Y_i^*F_i(Y^*)=0, \qquad 0\le i\le 3.
\end{equation}
One  always has the trivial equilibria
$$
E_1:=\mathbf{0}\in \mathcal{E}(p)
$$
and  the healthy equilibrium state
$$E_2\in \mathcal{E}(p),
$$
so that $\mathcal{E}(p)$ is always nonempty. The lemma below show that the value of the susceptible class $Y_0^*$ for the healthy equilibrium state $E_2$ is the largest possible among all equilibrium points  $Y^*$.

\begin{lemma}\label{lem:equi}
If $Y^*\ne\mathbf{0}$ is an element of $\mathcal{E}(p)$ then
\begin{equation}\label{Ystar}
0<Y^*_0\le S_2,
\end{equation}
where the (above) equality holds if and only if $Y_1^*=Y_2^*=Y_3^*=0$. Furthermore,
\begin{equation}\label{Kineq}
\sigma_1\le Y^*_0\le \min\{S_2,\sigma_3\},
\end{equation}
unless $Y^*=(S_2,0,0,0)$.
Also the following balance relations hold:
\begin{eqnarray}
\alpha_1Y^*_1+\alpha_2Y^*_2+\alpha_3Y^*_{3}&=&\frac{b}{K}(S_2-Y^*_0)
\label{law1}\\
\mu_1Y^*_1+\mu_2Y^*_2+\mu_3Y^*_{3}&=&\frac{b}{K}(S_2-Y_0^*)Y^*_0.
\label{law2}
\end{eqnarray}
In particular,
\begin{equation}\label{max}
\max_{0\le i\le 3}Y_i^*\le \frac{b}{K}\max\{\frac{1}{\alpha_1},\frac{1}{\alpha_2}, \frac{1}{\alpha_3},b-\mu_0\}
\end{equation}

\end{lemma}

\begin{proof}
Suppose first that $Y^*\ne\mathbf{0}$ and $Y_0^*=0$. If $Y_1^*\ne0$ then $Y_{3}^*=- \mu_1/\eta_1<0$, a contradiction. Therefore $Y_1^*=0$. For the same reason $Y_2^*=0$. Therefore it must be $Y_3^*\ne0$. But in that case, it follows from the last equation in \eqref{FY} by virtue of  $Y_1^*=Y_2^*=0$ that $-\mu_3=0$, a contradiction also. Therefore $Y_0^*\ne0$, thus it is positive,   which proves the left inequality in \eqref{Ystar}. Next, since $Y_0^*\ne0$, the relation \eqref{law1} follows immediately from the first equation in \eqref{FY}. Also, summing up all the four equations in \eqref{FY} yields \eqref{law2}. Next, since $Y^*_i\ge0$ it follows from \eqref{law2} that $S_2-Y_0^*\ge0$, which proves the second inequality in \eqref{Ystar}. Finally, if $Y^*\ge0$ then dividing \eqref{law2} by \eqref{law1} we obtain
\begin{equation}\label{YY0}
Y_0^*= \frac{\mu_1Y^*_1+\mu_2Y^*_2+\mu_3Y^*_{3}} {\alpha_1Y^*_1+\alpha_2Y^*_2+\alpha_3Y^*_{3}}.
\end{equation}
The latter expression is the ratio of two linear functions with positive coefficients. It is also zero degree homogeneous, hence its maximal/minimal values are attained at the simplex $\Pi:=\alpha_1Y^*_1+\alpha_2Y^*_2+\alpha_3Y^*_{3}=1$. It follows from the linearity of the numerator that
$$
\max_{\R{3}_+}
\frac{\mu_1Y^*_1+\mu_2Y^*_2+\mu_3Y^*_{3}} {\alpha_1Y^*_1+\alpha_2Y^*_2+\alpha_3Y^*_{3}}= \max_{\Pi}(\mu_1Y^*_1+\mu_2Y^*_2+\mu_3Y^*_{3})=\max\{\frac{\mu_1}{\alpha_1}, \frac{\mu_2}{\alpha_2}, \frac{\mu_3}{\alpha_3}\}=\sigma_3,
$$
and similarly
$$
\min_{\R{3}_+}
\frac{\mu_1Y^*_1+\mu_2Y^*_2+\mu_3Y^*_{3}} {\alpha_1Y^*_1+\alpha_2Y^*_2+\alpha_3Y^*_{3}}= \min_{\Pi}(\mu_1Y^*_1+\mu_2Y^*_2+\mu_3Y^*_{3})=\min\{\frac{\mu_1}{\alpha_1}, \frac{\mu_2}{\alpha_2}, \frac{\mu_3}{\alpha_3}\}=\sigma_1,
$$
which together with \eqref{YY0} and \eqref{Ystar} implies \eqref{Kineq}. Using \eqref{law1} one also easily obtains \eqref{max}.
\end{proof}

\section{The finiteness of $\mathcal{E}(p)$}

Following to \cite{Horn12} we recall some standard terminology. Given a quadratic matrix $A$, we denote  by $A[\alpha, \beta]$ the submatrix of entries that lie in the rows of $A$ indexed by $\alpha$ and the columns indexed by $\beta$. If $\alpha=\beta$, the submatrix is called principal. The corresponding determinant $\det A[\alpha,\alpha]$ is called the principal minor. An $n$-by-$n$ matrix has $\binom{n}{k}$ distinct principal submatrices of size $k$; i.e. totally, $2^n-1$ principal submatrices of order $1\le k\le n$.

Since the left hand side of \eqref{FY} is a quadratic polynomial map in $Y^*$, it follows from the standard algebraic geometry argument based on Bez\`{o}ut's theorem that \eqref{FY} has either (i) infinitely many or (ii) at most  $2^4=16$ distinct solutions, counting the trivial point $E_0:=\mathbf{0}$. A simple analysis shows that under condition \eqref{assum}, (i) is not possible. Indeed, we have the following  lemma which  can be justified by an elementary verification, but it has some several important implications.

\begin{lemma}\label{lem:minors}
Let $A=(a_{ij})_{0\le i,j\le 3}$ be the matrix in \eqref{gen11}. Then its determinant is
\begin{equation}\label{detA}
\det A=\Delta^2, \qquad \Delta:=\eta_1\alpha_2-\eta_2\alpha_1,
\end{equation}
and the only zero principal minors $\det A[\alpha,\alpha]$ are for
$$
\alpha\in \mathcal{G}:=
\{(0,1,2), \,\,(1,2,3), \,\, (1,2), \,\, (1), \,\,(2), \,\,(3)\}
$$
\end{lemma}

Let $\R{4}(\alpha)$ denote the subset
$$
\R{4}_+(\alpha)=\{x\in\R{4}_+:x_i=0 \text{ for all $i\in \alpha$}\}.
$$
For instance, $\R{4}_+(\emptyset)=\R{4}_+$ and $\R{4}_+(2,3)$ is the face consisting of the point with coordinates $(x_1,0,0,x_4)$, where $x_1,x_4\ge0$.
Given a subset $\alpha\subset \{1,2,3,4\}$ we denote by $\mathcal{E}(p,\alpha)$ the subset of $\mathcal{E}(p)\subset \R{4}_+(\alpha)$, and by $\bar \alpha$ we denote the complement $\bar\alpha=\{1,2,3,4\}\setminus \alpha$.

Here are some important observations following from Lemma~\ref{lem:minors}.

\begin{corollary}\label{cor:alpha}
If $\alpha=\emptyset$ then $\mathcal{E}(p,\emptyset)$ consists of at most one point when $\Delta\ne0$; if $\Delta=0$ then $\mathcal{E}(p,\emptyset)=\emptyset$. In particular, the number of equilibrium points in the interior $\mathrm{int}(\R{4}_+)$ is at most one.

\end{corollary}

\begin{proof}
Indeed, the only nontrivial part here is the  claim about the zero determinant (in this case, a priori maybe infinitely many solutions). To show that $\Delta=0$ implies $\mathcal{E}(p,\emptyset)=\emptyset$, we assume by contradiction that there is some $Y\in \mathcal{E}(p,\emptyset)$. Setting $\tau:=\eta_1/\alpha_1=\eta_2/\alpha_2$ one readily obtains from the second and the third equations in $AY=q$ that $Y_0-\tau Y_3=\frac{\mu_1}{\alpha_1}=\frac{\mu_2}{\alpha_2}$, which contradicts to \eqref{assum12}.
\end{proof}

\begin{corollary}\label{cor:le8}
$\mathrm{card}(\mathcal{E}(p))\le 8$.
\end{corollary}

\begin{proof}
By Bez\`{o}ut's theorem we have $\mathrm{card}(\mathcal{E}(p))\le 8$. Next, it is clear from \eqref{FY} and Corollary~\ref{cor:alpha} that for any admissible values of $p$ in \eqref{pparam} there can at most one equilibrium point  exist in $\mathrm{int}(\R{4}_+)$. Any other equilibrium points must have zero coordinates. Next, since by Lemma~\ref{lem:equi} $Y_0^*\ne0$ except $Y^*=\mathbf{0}$, at most  $8=1+3+3+1$ distinct nonnegative equilibrium points may exist.
\end{proof}

\section{Basic facts about the LCP}

An essential place in the further analysis plays the signs of $F_i(Y)$, where $Y$ is an equilibrium point of \eqref{gen}. In particular the situation when all coordinates are nonpositive is very distinguished. We have the definition.

\begin{definition}
An equilibrium point $Y\in \mathcal{E}(p)$  of \eqref{gen} is said to be $F$-\textit{stable} if $F_i(Y)\le 0$ for all $0\le i\le 3$.
\end{definition}

As we shall see below, if $p$ is admissible then there always exists a unique $F$-stable point. To prove the existence and uniqueness we employ the LCP (linear complementarity problem) machinery. An application of the LCP to Lotka-Volterra systems is not new and was firstly used by Takeuchi and Adachi\cite{Takeuchi80}, see also \cite{Takeuchibook}. On the other hand, in this paper we are interested  primarily  in a finer structure of the $F$-stable points, namely how this set depends on the fundamental parameters of the system. To proceed, we recall some basic facts about the linear complementarity problem.

The LCP (\textit{linear complementarity problem}) consists of finding a  vector in a finite dimensional real vector space that satisfies a certain system of inequalities. Specifically, given a vector $q\in \R{n}$ and matrix $M\in \R{n\times n}$, the LCP is to find a vector $z\in \R{n}$ such that
\begin{eqnarray}
  z &\ge& 0, \label{LCP1}\\
  q+Mz &\ge& 0, \label{LCP2}\\
  z^T(q+Mz) &=& 0.\label{LCP3}
\end{eqnarray}
We refer to \cite{Cottle} for a comprehensive account of the modern development of LCP. recall some standard terminology and facts following to \cite{Cottle}. A vector $z$ satisfying the inequalities \eqref{LCP1}, \eqref{LCP2} is called \textit{feasible}. Given a feasible vector $z$, let $$
w=q+Mz.
$$
Then $z$ satisfies \eqref{LCP3} if and only if $z_iw_i=0$ for all $i$.

The correspondence between the general LCP and our model is given by virtue of \eqref{gen11} and \eqref{FY0} as follows:
\begin{equation}\label{correspon}
\begin{split}
  z &\leftrightarrow  Y^*\\
  w &\leftrightarrow -F(Y^*)\\
  M &\leftrightarrow -A\\
  q &\leftrightarrow q.
\end{split}
\end{equation}
Indeed, we are  interested in \textit{nonnegative} equilibrium points, i.e. in those solutions of \eqref{FY0} which satisfy $Y_i\ge0$, which is exactly  condition \eqref{LCP1}. Furthermore, in this dictionary, \eqref{FY0} becomes equivalent to  equation \eqref{LCP3}. Finally, since by \eqref{gen10}
$$
q+Mz=q-AY=-F(Y)\ge 0,
$$
i.e. condition \eqref{LCP2} is equivalent to saying  that the corresponding equilibrium point $Y$ is $F$-stable.

In summary, we have

\begin{proposition}\label{pro:equiv}
$Y$ solves the LCP$(-A,q)$ if and only if $Y$ is an $F$-stable equilibrium point of \eqref{gen}.
\end{proposition}

%
%

The stability of \eqref{gen} depends on the number of possible $F$-stable points of our model. In general, the structure of LCP$(A,q)$ may be very arbitrary. In some cases depending on the matrix $M$, however, one have a more strong information. Therefore, in order to study this question we need to look at the matrix $A$ in \eqref{gen11} more attentively. To this end, we make the following important remark: since
$$
A+A^T=
\left(
 \begin{array}{cccc}
   -\frac{2b}{K}&0&0&0 \\
   0&0&0&0\\
   0&0&0&0 \\
   0&0&0&0
 \end{array}
 \right)
$$
it follows that $A$ is \textit{negative semi-definite} in the sense of quadratic forms. Then it follows from the general LCP theory that the following existence result holds for positive definite matrices:

\begin{proposition}[Theorem 3.1.6 in Cottle\cite{Cottle}]\label{pro:cottle1}
If a matrix $M$ is positive definite then the LCP$(q,M)$ has a unique solution for all $q\in \R{n}$.
\end{proposition}

In the next section, we shall apply a perturbation technique to utilize  Proposition ~\ref{pro:cottle1} to derive the existence of an $F$-stable point even for our semi-definite matrix $A$. In order to prove the uniqueness, we shall also need the following well-known result which proof we recall for the convenience reasons.

\begin{lemma}
\label{lem:convex}
Let $M$ be positive semi-definite in the sense of quadratic forms. Then the set of solutions of the LCP$(M,q)$ is convex.
\end{lemma}

\begin{proof}[Proof of Lemma~\ref{lem:convex}.]
Let $z$ and $\bar z$ be any two solutions of LCP$(M,q)$ and let $\zeta=az+b\bar z$, where $0\le a=1-b\le 1$. Then \eqref{LCP1} and \eqref{LCP2} obviously hold for $\zeta$. In order to verify \eqref{LCP3} we note that $z^Tw=\bar z^T\bar w=0$, where $w=q+Mz$ and $\bar w=q+M\bar z$. Hence
$$
-\bar z^Tw-z^T\bar w=(z-\bar z)^T(w-\bar w)=(z-\bar z)^TM(z-\bar z)\ge0.
$$
Since $\bar z^Tw\ge0$ and $z^T\bar w\ge0$, we conclude that actually the latter two inequalities are equalities, therefore $\bar z^Tw=z^T\bar w=0$.
\begin{align*}
\zeta^T(q+M\zeta)&=(az+b\bar z)^T(aw+b\bar w)\\
&=a^2z^Tw+b^2\bar z^T\bar w+ab(\bar z^Tw+z^T\bar w)\\
&=0,
\end{align*}
hence \eqref{LCP3} holds true for $\zeta$, and the lemma is proved.
\end{proof}

\section{The existence and uniqueness of an $F$-stable point}

We shall prove the main result (Theorem~\ref{theo:existandunique}) of this section.  The proof of existence and uniqueness of an $F$-stable point relies on the analysis of the associated linear complementarity problem for a perturbed system \eqref{gen}. We make an essential use of a special structure of the matrix $A$ in \eqref{gen11}. Note, however, that for a general positive semi-definite matrix $M$ the uniqueness of an $F$-stable point is failed.

\begin{theorem}
\label{theo:existandunique}
Let $A$ be the matrix in \eqref{gen11} and let $p$ be an admissible vector. Then there exists a unique  $F$-stable point of \eqref{gen}.
\end{theorem}

\begin{proof}
We consider a perturbation of \eqref{gen}. Let $M=I\varepsilon-A$, where $I\in \R{4\times 4}$ is the unit matrix. Then $M$ is positive definite for any $\varepsilon>0$. Let $z=z(\varepsilon)$ denote the unique solution accordingly to Proposition~\ref{pro:cottle1}. Then using the dictionary \eqref{correspon} we obtain
\begin{eqnarray}
z(\varepsilon)=(z_0(\varepsilon),z_1(\varepsilon),z_2(\varepsilon),z_3(\varepsilon)) &\ge& 0\label{wz0}\\
w(\varepsilon):=q+(I\varepsilon-A)z(\varepsilon)&\ge&0\label{wz1}\\
w_i(\varepsilon)z_i(\varepsilon)&=&0, \quad 0\le i\le 3.\label{wz2}
\end{eqnarray}

Our first claim is that $\max_{1\le i\le 3} \{z_i(\varepsilon)\}$ is uniformly bounded when $\varepsilon\to 0^+$.
We have
\begin{align*}
0&=\sum_{i=0}^3w_i(\varepsilon)z_i(\varepsilon)\\
&=\sum_{i=0}^3(q+(I\varepsilon-A)z_(\varepsilon))_iz_i(\varepsilon)\\
&=\frac{b}{K}z_0^2(\varepsilon)+\sum_{i=0}^3q_iz_i(\varepsilon)+\varepsilon z_i(\varepsilon)^2,
\end{align*}
hence we find from \eqref{gen11} that
\begin{equation}\label{bmu}
(b-\mu_0)z_0(\varepsilon)= \frac{b}{K}z_0^2(\varepsilon)+\sum_{i=1}^3\mu_iz_i(\varepsilon)+\varepsilon \sum_{i=0}^3z_i(\varepsilon)^2.
\end{equation}
Since the sums in the right hand side are nonnegative, we have $(b-\mu_0)z_0(\varepsilon)\ge \frac{b}{K}z_0^2(\varepsilon)$, thus
\begin{equation}\label{bmu1}
z_0(\varepsilon)\le \frac{K}{b}(b-\mu_0)=S_2,
\end{equation}
i.e. $z_0(\varepsilon)$ is uniformly bounded when $\varepsilon\to 0^+$.
Using this in \eqref{bmu} yields
\begin{equation}\label{estim1}
\begin{split}
\mu \max\{z_1(\varepsilon),z_2(\varepsilon),z_3(\varepsilon)\}&\le \sum_{i=1}^3\mu_iz_i(\varepsilon)\\
&\le \sum_{i=1}^3\mu_iz_i(\varepsilon)+\varepsilon \sum_{i=0}^3z_i(\varepsilon)^2 +\frac{b}{K}z_0^2(\varepsilon)\\
&= (b-\mu_0)z_0(\varepsilon)\\
&\le (b-\mu_0)S_2,
\end{split}
\end{equation}
where $\mu:=\min\{\mu_1,\mu_2,\mu_3\}$. Therefore,
\begin{equation}\label{limsup}
\max\{z_1(\varepsilon),z_2(\varepsilon),z_3(\varepsilon)\} \le \frac{(b-\mu_0)S_2}{\mu}
\end{equation}
hence  the first claim follows from \eqref{bmu1} and \eqref{limsup}.

Now, with the boundedness in hands, we conclude that there exists a sequence $\varepsilon_j\to0^+$ such that $z(\varepsilon_j)$ converges, say
$$
\lim_{\varepsilon_j\to0^+}z(\varepsilon_j)=Y:=(Y_0,Y_1,Y_2,Y_3).
$$ Then  for continuity reasons we have
\begin{eqnarray}
Y &\ge& 0\label{Yz0}\\
F(Y)=-q+AY&\le&0\label{Yz1}\\
Y_iF_i(Y)&=&0, \quad 0\le i\le 3.\label{Yz2}
\end{eqnarray}
Therefore $Y$ is an $F$-stable equilibrium point of \eqref{gen}.

Our next claim is that there thus obtained $F$-stable equilibrium point is unique. In order to prove this, note that by Lemma~\ref{lem:convex} the set of $F$-stable equilibrium points is convex. Suppose that $Y\ne Y'$ are two $F$-stable equilibrium points of \eqref{gen}. Then the segment between $Y$ and $Y'$ consists of $F$-stable equilibrium points. In other words, all points $Y'=Y+vt$, where $0\le t\le 1$ and $v=Y'-Y$ are $F$-stable equilibrium points. First note that $Y_0=Y_0'$. Indeed,  applying \eqref{law2} to $Y^*=Y+vt$ and differentiating twice the obtained identity  with respect to $t$ we obtain $-\frac{2b}{K}v_0^2=0$, hence $v_0=Y'_0-Y_0=0$.

All other three coordinates of the segment are a linear functions of $t$: $Y_i^*(t):=Y_i+(Y_i'-Y_i)t$, $1\le i\le 3$,  hence they are either identically zero or have at most one zero. Therefore, modifying, if needed the ends $Y$ and $Y'$, we may assume that for each $i$ exactly one condition holds: (a) either $Y_i^*(t)\equiv 0$, or (b) $Y_i^*(t)\ne0$ for all $0\le t\le 1$. Note also that at least one coordinate $i$ must satisfy (a). Indeed, by the first claim of  Corollary~\ref{cor:alpha}, the number of equilibrium points in the interior $\mathrm{int}(\R{4}_+)$ is at most one, therefore, for continuity reasons, none of $Y$ and $Y'$ can lie in $\mathrm{int}(\R{4}_+)$.

Thus, the above observations imply that $Y$ and $Y'$ must lie on the same face. Since by Lemma~\ref{lem:equi} $Y_0=Y_0'\ne0$, the face equation must be $\{Y_k=0:k\in K\}$ for some (nonempty!) subset $K\subset \{1,2, 3\}$. On the other hand, the nonzero coordinates $Y_i^*(t)$ must satisfy \eqref{law1}--\eqref{law2}. Since the latter equations are linearly independent by \eqref{assum12}, and the number of nonzero $Y_i^*(t)$ is $\le 3-1=2$ (at least one must satisfy the condition (a)!), we conclude that there exists at most one solution. This contradicts to the infinitely many points in the segment between $Y$ and $Y'$, and thus finishes the proof of the uniqueness.
\end{proof}

\section{A finer structure of $\mathcal{E}(p)$}

In what follows, we are interested in the equilibrium points with non-negative coordinates only. According to Corollary~\ref{cor:le8}, the set of equilibrium points is finite (there are at most 8 distinct points in $\R{4}_+$). Thus, to find which of these points is actually $F$-stable, is the choice problem: it suffices to check that the corresponding $F$-coordinates are nonpositive. Note that by Theorem~\ref{theo:existandunique} such a point must be unique! We make this analysis below.

Let $p$ be an admissible parameter vector and let $Y^*=Y^*(p)$ denote the unique  $F$-stable equilibrium point of \eqref{gen}. It is easily to see that  the trivial equilibrium $\mathbf{0}$ is never $F$-stable, i.e. the origin is the extinction equilibrium. Thus, by \eqref{Ystar}
$$
0<Y_0^*=Y_0^*(p)\le S_2.
$$

 The identically zero coordinates of an equilibrium point is called its \textit{zero pattern}. It follows from the  structure properties of the matrix $A$ that \textit{if $p$ is admissible then  there can exist at most one point with a given zero pattern}. A simple inspection yields the following nontrivial equilibrium points:
\begin{equation}\label{Ediag}
\begin{array}{llllll}
E_2&=(S_2,&0,&0,&0&)\\
E_3&=(S_3,&(S_2-S_3)\frac{b}{K\alpha_1},&0,&0&)\\
E_4&=(S_4,&0,&(S_2-S_4)\frac{b}{K\alpha_2},&0&)\\
E_5&=(S_5,&0,&0,&(S_2-S_5)\frac{b}{K\alpha_3}&)\\
E_6&=(S_6,&(S_5-S_6)\frac{\alpha_3}{\eta_1},&0, &(S_6-S_3)\frac{\alpha_1}{\eta_1}&)\\
E_7&=(S_7,&0,&(S_5-S_7)\frac{\alpha_3}{\eta_2}, &(S_7-S_4)\frac{\alpha_2}{\eta_2}&)\\
E_8&=(S_8,&(S_8-S_7)\frac{b\eta_2}{K\Delta},
&(S_6-S_8)\frac{b\eta_1}{K\Delta} & (S_4-S_3)\frac{\alpha_1\alpha_2}{\Delta}&),\\
\end{array}
\end{equation}
where
$$
\Delta=\alpha_2\eta_1-\alpha_1\eta_2
$$
and $S_k=S_k(p):=(E_k)_0$ are the susceptible coordinates of the corresponding equilibrium state $E_k$  given respectively by
\begin{align}
S_3&=\sigma_1<S_4=\sigma_2<S_5=\sigma_3\label{e345}\\
(S_2-S_6)\frac{b}{K\alpha_1}&=(S_5-S_3)\frac{\alpha_3}{\eta_1}, \label{e6}\\
(S_2-S_7)\frac{b}{K\alpha_2}&=(S_5-S_4)\frac{\alpha_3}{\eta_2}, \label{e7}\\
S_8&=\frac{\delta}{\Delta}, 
\label{e8}\\
\delta&:=\mu_2\eta_1-\mu_1\eta_2.\nonumber
\end{align}
Note that modulo \eqref{e345}, the formulae \eqref{e6} and \eqref{e7} define explicitly $S_6$ and $S_7$ respectively.
We also  emphasize that $E_8$ exists (but maybe lie outside $\mathcal{E}(p)$) if and only if $\Delta=\alpha_2\eta_1-\alpha_1\eta_2\ne0$ (cf. with Corollary~\ref{cor:alpha}).

The equilibrium point $E_2$ is the disease free equilibrium, while the remaining equilibria $E_k$, $k\ge3$ are all endemic equilibria.

The above points $E_k$  (except for $E_8$) are well defined for all values of parameters, they can lie or not in $\R{4}_+$, but only one of them is $F$-stable. The latter, however, must be understood in the sense that for certain values of parameter $p$ it may happen that two different \textit{notations} $E_k$ coincide as points, for example $E_3(p)=E_6(p)$. We discuss this in more details below in Section~\ref{sec8}.

\begin{remark}\label{rem2}
Note that by \eqref{assum}
$$
\delta-\sigma_1\Delta=(\sigma_2-\sigma_1)\alpha_2\eta_1>0,
$$
in particular, $\delta$ and $\Delta$ cannot vanish simultaneously.
\end{remark}

Note also that the parameters $S_2,\ldots, S_8$ are dependent. On the other hand, we want to keep $S_2$ as a fundamental parameter of the model (the modified carrying capacity), and also consider $S_3,S_4$ and $S_5$ as the fundamental parameters satisfying the constraint \eqref{assum12}. Then it is convenient to think of $S_5,S_7$ and $S_8$ as depending on the first four fundamental parameters. It worthy to mention also that one has  from  \eqref{e6}-\eqref{e7} the following additional relations:
\begin{align}\label{e67}
\sigma_3\Delta-\delta=(S_5-S_8)\Delta&=(S_6-S_7)\frac{b\eta_1\eta_2}{\alpha_3 K},\\
\delta-\sigma_1\Delta=(S_8-S_3)\Delta&=(S_4-S_3)\eta_1\alpha_2,\label{e83}\\
\delta-\sigma_2\Delta=(S_8-S_4)\Delta&=(S_4-S_3)\eta_2\alpha_1.\label{e84}
\end{align}
\noindent
Note that these formulae are  well-defined even if $E_8$ does not exist (i.e. $\Delta=0$).

To study the $F$-stability we also write down the corresponding $F$-parts:
\begin{equation}\label{Fdiag}
\begin{array}{llllll}
F(E_2)&=(0,&(S_2-S_3)\alpha_1,&(S_2-S_4)\alpha_2, &(S_2-S_5)\alpha_3&)\\
F(E_3)&=(0,&0,&(S_3-S_4)\alpha_2,&(S_6-S_3)\frac{\alpha_3b\eta_1}{K\alpha_1}&)\\
F(E_4)&=(0,&(S_4-S_3)\alpha_1,&0,&(S_7-S_4)\frac{\alpha_3b\eta_2}{K\alpha_2}&)\\
F(E_5)&=(0,&(S_5-S_6)\frac{b\eta_1}{K\alpha_3}, &(S_5-S_7)\frac{b\eta_2}{K\alpha_3},&0&)\\
F(E_6)&=(0,&0,&(S_6-S_8)\frac{\Delta}{\eta_1}, &0&)\\
F(E_7)&=(0,&(S_8-S_7)\frac{\Delta}{\eta_2},&0, &0&)\\
F(E_8)&=(0,&0,&0 &0&)\\
\end{array}
\end{equation}

Using the obtained relation and the existence/uniqueness result, one may easily  by inspection to find which of the seven points $E_i$ is $F$-stable for a given $p$. It is rather trivial task for a concrete value of $p$, but, of course, an explicit description of $k(p)$, where $E_{k(p)}$ is $F$-stable, is a more nontrivial problem. Still, it is possible to get some simple conditions to outline the main idea.

\begin{proposition}\label{proE2}
The following $F$-stability conditions holds:
\begin{enumerate}
\item[(i)]the point $E_2$ is $F$-stable if and only if
$
S_2\le \sigma_1,
$
i.e. when the carrying capacity is small enough;
\item[(ii)]
the point $E_3$ is $F$-stable if and only if
$
S_6\le \sigma_1\le S_2;
$
\item[(iii)]
the point $E_5$ is $F$-stable if and only if
$$
\sigma_3\le \min\{S_2,S_6,S_7\};
$$
\item[(iv)]
the point $E_6$ is $F$-stable if and only if
\begin{align*}
(S_6-S_8)\Delta&\le0,\\
\sigma_1\le S_6&\le \sigma_3,
\end{align*}
\item[(v)]
the point $E_7$ is $F$-stable if and only if
\begin{align*}
(S_8-S_7)\Delta&\le0,\\
\sigma_2\le S_7&\le \sigma_3,
\end{align*}
\item[(vi)]
the point $E_8$ is $F$-stable if and only if
\begin{align*}
\Delta&>0,\\
\max\{0,S_7\}\le S_8&\le S_6
\end{align*}
\end{enumerate}
In the borderline cases (when some inequality becomes an equality), the corresponding equilibrium points coincide; for example, if $S_2=\sigma_1$ then $E_2=E_3$.
\end{proposition}

\begin{proof}
First, it easily follows from \eqref{Fdiag} that $E_2\ge0$ always, while $F(E_2)\le0$ if and only if $S_i\ge S_2$ for all $i=3,4,5$. By the uniqueness of an $F$-stable point, this immediately implies (coming back to the $\sigma$-notation in \eqref{e345}) that (i) holds.
Similarly, $E_3\ge0$  if and only if  $S_2-S_3\ge0$, i.e. $S_2\ge \sigma_1$. On the other hand, since $(S_3-S_4)\alpha_2=(\sigma_1-\sigma_2)\alpha_2<0$, we see that $F(E_3)\le0$ is equivalent to inequality $S_6-S_3\le0$, i.e. $S_6\le \sigma_1$. This implies (ii). Analysis of (iii)-(v) is similar. Finally, analysis of $E_8$ reduces to the nonnegativity of its coordinates. The last coordinate must be nonnegative, hence (by virtue of $S_4-S_3=\sigma_2-\sigma_1>0$) we must have $\Delta>0$. This readily yields the desired inequalities.
\end{proof}

We summarize the above observations be some remarks. According to what was done before, we a priori know that the conditions of Proposition~\ref{proE2} are complementary to each other in the sense that they have no common (interior) points and give together the whole set of admissible parameters. This, however, is very difficult to see from the explicit defining inequalities. One reason for that is that the parameters $S_k$, $k=6,7,8$ are dependent on the fundamental parameters.

Also, it is not a priori clear that any of the conditions in Proposition~\ref{proE2} are realizable for some $p$. In fact, it is an elementary exercise to verify that any of the $E_k$, $k\in \{2,3,5,6,7,8\}$ may be realizable for some admissible $p$. The reader can easily verify this by expanding the explicit values for $S_k$, $k=6,7,8$ in the above inequalities, but we do not give these rather cumbersome expressions. Instead, a more important question is to study the dependence of the $F$-stable point on some distinguished parameters like tha carrying capacity $K$. We consider this problem in more details below in Section~\ref{sec8}.

Finally, as for many epidemiology models, the above results could also be interpreted as the threshold in terms of the basic reproduction number
$R_0$, which is usually defined as the average number of secondary infections produced when one infected individual is introduced into a host population where everyone is susceptible \cite{HDietz}. In the context of the present paper, the most natural definition of the basic reproduction number for a virus would be similar to that considered by Allet et al in \cite{Allen}. It follows also that there are additional threshold values which depend on the dynamics of the population size at the equilibrium values, see the discussion, cf. \cite[p.~198]{Allen}.

\section{The global stability}

Now connect the concept of the $F$-stability to the Lyapunov stabilty.
Recall that an equilibrium point $Y^*=(Y_0^*,Y_1^*,Y_2^*,Y_3^*)\in \mathcal{E}(p)$ is called $F$-stable if $Y_i^*\ge0$ and $F_i(Y^*)\le 0$ for any $0\le i\le 3$. An $F$-stable point $Y^*$ is said to be \textit{degenerate} if $Y_i^*=F_i(Y^*)=0$ for some $0\le i\le 3$. In other words, an equilibrium point $Y^*$ is degenerate if the \textit{total number} of nonzero coordinates of both $Y^*$ and $F(Y^*)$ is less than $4$.

The above terminology can be motivated by  the following observation. Given $Y^*\in \mathcal{E}(p)$, we associate the generalized Volterra function\cite{Plank}
$$
V_{Y^*}(y_0,y_1,y_2,y_3)=\sum_{i=0}^3 (y_i-Y_i^*\ln y_i).
$$
Then the time derivative of $V_{Y^*}$ along any integral trajectory of \eqref{gen} is given by
\begin{equation}\label{Fequa}
\frac{d}{dt}V_{Y^*}:= (\nabla V_{Y^*})^T\,\frac{dy}{dt}=-\frac{b}{K}(y_0-Y_0^*)^2+\sum_{i=0}^3 F_i(Y^*)y_i.
\end{equation}
Therefore, \textit{if $Y^*$ is an $F$-stable point of \eqref{gen} then it is Lyapunov stable}:
\begin{equation}\label{Lyap}
\frac{d}{dt}V_{Y^*}(y(t))\le 0.
\end{equation}

The following elementary observation is a useful tool to sort away certain $F$-stable points.
\begin{proposition}\label{rem1}
The equilibrium points $E_1=\mathbf{0}$  and $E_4$ are never $F$-stable.
\end{proposition}

\begin{proof}
Indeed, $F(E_1)_1=b-\mu_0>0$ and $(F(E_4))_2=\alpha_1(\sigma_2-\sigma_1)>0.$
\end{proof}

Our  principal result establishes the existence and uniqueness of an $F$-stable point.

\begin{theorem}
\label{th1}
The $F$-stable point $Y^*(p)$ is globally stable, i.e. $Y(t)\to Y^*(p)$ as $t\to \infty$ for any solution of \eqref{model} with initial data \eqref{initdata}. Furthermore,
$$
0<\min\{S_2,\sigma_1\} \le Y_0^*(p)\le \min\{S_2,\sigma_3\}.
$$
In particular, $Y_0^*(p)\le \sigma_3$ with the equality if and only if $Y^*(p)=E_5$.
\end{theorem}

\begin{remark}
Note, however, that an explicit representation and the zero pattern of the corresponding $F$-stable point $Y^*(p)$ depends in a tricky way on the fundamental parameter $p$. The proof of the global stability makes an essential use of the fact that $Y_0(t)$ has a nonzero limit value. This allows us to obtain  nontrivial  first integrals which reduce the dimension of the $\omega$-limit set to $0.$
\end{remark}

\begin{remark}\label{reminf}
The asymptotic behaviour of \eqref{model} maybe, however, rather complex if $K=\infty$ and will be treated somewhere else. Note also that if $K=\infty$, the system \eqref{model} is no longer semi-definite but it has the pure skew-symmetric structure instead. More precisely, the matrix $A$ in \eqref{gen11} is skew-symmetric and can be thought of as a  perturbation of the decomposable matrix
$$
B=  \left(
 \begin{array}{cccc}
   0&-\alpha_1&0&0 \\
   \alpha_1&0&0&0\\
   0&0&0&-\eta_2 \\
   0&0&\eta_2&0
 \end{array}
 \right).
$$
The dynamic of perturbed Lotka-Volterra systems obtained by perturbation of $B$ can be very complex and contain nontrivial attractors in $\R{4}$, as the recent results of \cite[Part II]{Duarte98} show.
\end{remark}

We begin by proving some auxiliary statements.

\begin{proposition}
If $Y(t)$ is a solution of \eqref{gen} satisfying \eqref{initdata} then
	\begin{equation}\label{colo}
	Y_0(t) \leq \left(\frac{1}{S_2}(1-e^{-(b-\mu_0)t})+\frac{1}{Y_0(0)}e^{-(b-\mu_0)t}\right)^{-1}.
	\end{equation}
	In particular,
	\begin{equation}\label{boundr1}
	 Y_0(t)\leq \max\{ S_2,Y_0(0)\}
	\end{equation}
	and
	\begin{equation}\label{limit}
	\limsup_{t\rightarrow \infty} Y_0(t)\leq  S_2.
	\end{equation}
\end{proposition}
\begin{proof}
It follows from the first equation of \eqref{gen} that
\begin{equation*}
Y_0'-(b-\mu_0)Y_0 \leq-\frac{bY_0^2}{K},
\end{equation*}
which can be written as
\begin{equation*}
(Y_0e^{-(b-\mu_0)t})' \leq - \frac{b}{K} e^{-(b-\mu_0)t}Y_0^2.
\end{equation*}
Integrating the latter inequality gives
\begin{equation*}
\frac{e^{(b-\mu_0)t}}{Y_0} \geq \frac{b}{K(b-\mu_0)} (e^{(b-\mu_0)t}-1) + \frac{1}{Y_0(0)},
\end{equation*}
which proves \eqref{colo}. Relations \eqref{boundr1} and \eqref{limit} are direct consequences of \eqref{colo}.
 \end{proof}

 \begin{proposition}\label{pro:Ybound}
 	If $Y(t)$ is a solution of \eqref{gen} with \eqref{initdata} then
 	\begin{equation}
 	\sum_{i=0}^3Y_i(t) \leq \max\{\sum_{i=0}^3Y_i(0), \,\frac{Kb}{4\hat\mu}\}
 	\end{equation}
 	for $t \geq 0$, where $\hat\mu:=\min\{\mu_0,\mu_1,\mu_2,\mu_3\}$. In particular, any solution of \eqref{gen} with initial data  \eqref{initdata} is bounded.
 \end{proposition}
\begin{proof}
Summing up all equations of \eqref{gen} gives
\begin{equation*}
\begin{split}
\frac{d}{dt}\sum_{i=0}^3Y_i(t)&=\frac{b}{K}(S_2-Y_0)Y_0- \sum_{i=1}^3\mu_iY_i(t)\\
&\leq \frac{b}{K}(K-Y_0)Y_0-\hat\mu \sum_{i=0}^3Y_i(t).
\end{split}
\end{equation*}
Setting  $f(t)=\sum_{i=0}^3Y_i(t)$ we find $f'(t) \leq \frac{bK}{4}-\hat\mu f(t).$ By integrating the above equation  we obtain the desired inequality.
\end{proof}

\begin{proof}[Proof of Theorem~\ref{th1}]
According to Theorem~\ref{theo:existandunique} there exists a unique $F$-stable point, we denote it by $Y^*$. Let $Y(t)$ be any solution of \eqref{gen} with initial conditions \eqref{initdata}. First note that by \eqref{Fequa}, the Volterra function $V_{Y^*}(t)$ is nonincreasing for all $t\ge0$, therefore
\begin{equation}\label{ineqV0}
V_{Y^*}(Y(t))\le V_{Y^*}(0).
\end{equation}
On the other hand, $V_{Y^*}(t)$ is a priori bounded from below. Indeed, it is easily verified that the function of one variable $\psi_a(x)=x-a\ln x$ is decreasing in $(0,a)$ and increasing for $x\in (a,\infty)$, thus
$$
\psi_a(x)=x-a\ln x\ge \psi_a(a)=a-a\ln a, \qquad x\in (0,\infty)
$$
(where the above inequality also holds in the limit case $a=0$). It follows that
\begin{equation}\label{ineqV1}
V_{Y^*}(y)\ge V_{Y^*}(Y^*),
\end{equation}
and the equality holds if and only if $y=Y^*$. Thus, $V_{Y^*}(Y(t))$ is uniformly bounded in $\R{4}_+$. Coming back to  \eqref{Fequa}, note that by our choice of $Y^*$, all $F_i(Y^*)\le0$ and also $Y_i\ge0$, therefore for any $T>0$:
$$
\int_{0}^T\frac{b}{K}(Y_0(t)-Y_0^*)^2\,dt+\sum_{i=0}^3 |F_i(Y^*)| \int_{0}^T|Y_i(t)|\,dt=V_{Y^*}(Y(0))-V_{Y^*}(Y(T)).
$$
This immediately implies by the uniform boundedness of $V_{Y^*}(Y(t))$ that

\begin{itemize}
\item[(a)]
the function $Y_0(t)-Y_0^*\in L^2([0,\infty))$;
\item[(b)]
if $F_i(Y^*)\ne0$ (i.e. $F_i(Y^*)<0$) then the function $Y_i(t)\in L^1([0,\infty))$.
\end{itemize}

We have for any $0\le i\le 3$ that $Y_i(t)\ge0$  and by Proposition~\ref{pro:Ybound} $Y_i(t)\le M_i<\infty$ for all $t\ge0$. Therefore, it follows from  \eqref{gen} that for each fixed $i$ the derivative $Y_i'(t)$ is uniformly bounded on $[0,\infty)$. Differentiating any equation in \eqref{gen}, we conclude by induction that
\begin{equation}\label{conclude}
\text{\textit{for any $0\le i\le 3$, all derivatives $Y_i'(t),Y_i''(t),\ldots,Y_i^{(k)}(t),$ of any order $k\ge1$ are unifromly bounded in $[0,\infty)$}.}
\end{equation}

Combining \eqref{conclude} and (a) with Lemma~\ref{lem:Lpbis} in Appendix, we conclude that  the following limit exists:
\begin{equation}\label{Y0lim}
\lim_{t\to\infty}Y_0(t)=Y_0^*.
\end{equation}
Recall also that since $Y_0^*\ne0$ then
\begin{equation}\label{FFF}
F_0(Y^*)=0.
\end{equation}
Next, let $I$ denote the subset of $\{1,2,3\}$ such that  $F_i(Y^*)\ne0$ for some $i\in I$. Then by \eqref{FY}, $Y^*_i=0$, and, on the other hand by (b) we have $Y_i(t)\in L^1([0,\infty))$. Applying Lemma~\ref{lemma1} we find
\begin{equation}\label{Yilim}
\lim_{t\to\infty}Y_i(t)=0=Y_i^*, \qquad i\in I.
\end{equation}

It remains to establsh that $Y_j(t)$ converges also for $j\in J=\{1,2,3\}\setminus I$. Alternatively,
$$
J=\{j\ge1:F_j(Y^*)=0\}.
$$
Arguing as above and combining \eqref{conclude} with Corollary~\ref{cor:Lpbis}, we obtain
\begin{equation}\label{Y0k}
\lim_{t\to\infty} \frac{d^k}{dt^{k}}Y_0(t)=0\quad \text{ for any } k=0,1,2,\ldots.
\end{equation}
We have by \eqref{Y0k} that
\begin{equation}\label{diff1}
\lim_{t\to\infty} \frac{d}{dt}Y_0(t)=\lim_{t\to\infty}Y_0(t)F_0(Y(t))=0.
\end{equation}
Since $\lim_{t\to\infty}Y_0(t)=Y_0^*\ne0$, we obtain that
\begin{equation}\label{F0Yt}
\lim_{t\to\infty}Y_0(t)F_0(Y(t))=0,
\end{equation}
hence
\begin{equation}\label{diff2}
\lim_{t\to\infty} \sum_{i=1}^3\alpha_i Y_i(t)=\frac{b}{K}(S_2-Y_0^*).
\end{equation}
Since we also know that \eqref{Yilim} holds true for any $i\in I$, we my simplify \eqref{diff2} to obtain
\begin{equation}\label{diff2j}
\lim_{t\to\infty} \sum_{j\in J}\alpha_j Y_j(t)=\frac{b}{K}(S_2-Y_0^*).
\end{equation}
On the other hand, since $V_{Y^*}(Y(t))$ is nonincreasing and bounded, we similarly obtain
\begin{equation}\label{diff3}
\lim_{t\to\infty} \sum_{j\in J} (Y_j(t)-Y_j^*\ln Y_j(t))=C,
\end{equation}
where $C$ is some real constant.

To proceed, we iterate \eqref{Y0k} by virtue of \eqref{gen}. For example, the second derivative is obtained by
\begin{align*}
\frac{d^2}{dt^2}Y_0(t)&=\frac{d}{dt}(Y_0(t)F_0(Y(t))=Y_0'F_0(Y)+Y_0 \sum_{i=0}^3\frac{\partial F_0}{\partial Y_i}Y'_i(t)\\
&=Y_0\left(F_0^2(Y)+ \sum_{i=0}^3 Y_iF_i\frac{\partial F_0}{\partial Y_i}\right)=Y_0\mathcal{L}(F_0),
\end{align*}
where $\mathcal{L}$ is a Riccati type operator
$
\mathcal{L}(g)=g^2+ \sum_{i=0}^3 Y_iF_i\frac{\partial g}{\partial Y_i}.
$
Then \eqref{Y0k} and \eqref{Y0lim} imply that
$$
\lim_{t\to\infty} \mathcal{L}^k(F_0)(Y(t))=0, \quad \text{for all $k=0,1,2,\ldots$}
$$
For example, $k=1$ yields by virtue of \eqref{Y0lim}, \eqref{Yilim}, \eqref{FFF} and \eqref{F0Yt} that
\begin{equation}\label{Jeq}
\lim_{t\to \infty}\sum_{j\in J} \alpha_j Y_j(t)F_j(Y(t))=0.
\end{equation}

Next, note that if the cardinality of $J$  is exactly one then the left hand side of \eqref{diff2j} contains only one term, thus implying  the convergence of the corresponding $J$-coordinate. Therefore, we may assume without loss of generality that $J$ contains at least two indices.

We consider the two cases.

\textbf{Case 1.} Let $J$ be maximal possible, i.e. $J=\{1,2,3\}$. Then it must be
$$
F(Y^*)=(0,0,0,0).
$$
we have from \eqref{diff2}, \eqref{Y0lim}, \eqref{Jeq}, \eqref{diff3} and explicit expressions for $F_i$ that
 \begin{align}\label{Jeq1}
\lim_{t\to \infty}G(Y(t))&=\frac{b}{K}(S_2-Y_0^*),\\
\lim_{t\to \infty}H(Y(t))&=0,\label{Jeq2}\\
\lim_{t\to \infty}V_{Y^*}(Y(t))&=C,\label{Jeq3}
\end{align}
where $G(y)=\sum_{i=1}^3\alpha_iy_i,$ and
$
H(y)=\sum_{i=1}^3\alpha_i c_i y_i+y_1y_3(\alpha_3-\alpha_1)\eta_1 +y_2y_3(\alpha_3-\alpha_2)\eta_2$, and $c_i=\alpha_i (Y_0^*-\sigma_i)$.
Then \eqref{Jeq1}--\eqref{Jeq3} implies that the $\omega$-set of the trajectory $Y(t)$ is a subset of the variety defined by
\begin{align}\label{Jeq10}
G(y_1,y_2,y_3)&=\frac{b}{K}(S_2-Y_0^*),\\
H(y_1,y_2,y_3)&=0,\label{Jeq20}\\
V_{Y^*}(y_1,y_2,y_3)&=C,\label{Jeq30}
\end{align}
We claim that the latter system has only finitely many solutions in $\R{3}_+$. Indeed, the left hand sides of \eqref{Jeq1}--\eqref{Jeq2}  are algebraic polynomials of degree 1 and at most 2 respectively.  Therefore, \eqref{Jeq1}--\eqref{Jeq2} defines either a curve of order two, or a line, or a plane. The latter is, however, possible only if the linear form $\phi:=\sum_{i=1}^3\alpha_iy_i-\frac{b}{K}(S_2-Y_0^*)$ divides $H$. Let us show that the latter is impossible. Indeed, suppose that $\phi$ divides $H$, then must exist a linear function $P$ of $y$ such  that
$$
\sum_{i=1}^3\alpha_i c_i y_i+y_1y_3(\alpha_3-\alpha_1)\eta_1 +y_2y_3(\alpha_3-\alpha_2)\eta_2= (\sum_{i=1}^3\alpha_iy_i-\frac{b}{K}(S_2-Y_0^*))P.
$$
On substitution $y_1=y_2=0$ into the latter identity we obtain
$$
\alpha_3c_3 y_3= (\alpha_3y_3-\frac{b}{K}(S_2-Y_0^*))P(0,0,y_3),
$$
therefore $S_2=Y_0^*$. But we know by Lemma~\ref{lem:equi} that the latter holds if and only if $Y^*=E_2=(S_2,0,0,0)$ in which case we have
$$
F(Y^*)=F(E_2)=(0,(S_2-\sigma_1)\alpha_1,(S_2-\sigma_2)\alpha_2, (S_2-\sigma_3)\alpha_3),
$$
see \eqref{Fdiag}. But by \eqref{assum} there exist at least two nonzero coordinates in $F(Y^*)$, a contradiction with the initial assumption. This proves that \eqref{Jeq1}--\eqref{Jeq2} define either a curve of order two or a straight line. Next, since at least one of $Y^*_i$ is nonzero for $i\ge1$ (because $Y^*\ne E_2$), it follows that eq.~\eqref{Jeq30} is transcendent (contains a logarithm). A simple argument show that in that case \eqref{Jeq1}--\eqref{Jeq3} must have at most finitely many (more precisely, $\le 6$) solutions. Thus the $\omega$-set is finite, implying for continuity reasons that the $\omega$-set is a point, i.e. all three limits $\lim_{t\to\infty}Y_i(t)$, $1\le i\le 3$, must exist. Then a standard argument reveals that the only possibility here is that the limit point is $Y^*$.

\textbf{Case 2.}  It remains to consider the case when the cardinality is exactly two, i.e. $J$ is obtained by eliminating some index $i\in\{1,2,3\}$. Write this as $J=\{j,k\}$ such that $\{1,2,3\}=\{i,j,k\}$. By the made assumption, $F_i(Y^*)\ne0$, $\lim_{t\to\infty}Y_i(t)=0=Y_i^*$, and
$$
F_j(Y^*)=F_k(Y^*)=0.
$$
Again, eliminating the trivial case $Y^*=E_2$ we may assume that at least one of coordinates, say $Y^*_j$, is nonzero. This implies that $V_{Y^*}(y_j,y_k)$ is a transcendent function. Repeating the argument in Case~1 we again arrive to the finiteness of the $\omega$-set, implying the convergence of $Y(t)$ to $Y^*$. The theorem is proved completely.



\end{proof}

\section{Transition dynamics of an $F$-stable point}\label{sec8}
From the biological point of view, it is important to know how the dynamics of the $F$-stable equilibrium point $Y^*(p)$ depends on the fundamental parameter $p\in \R{11}$. We have the following general result.

\begin{theorem}\label{th:cont}
The map $p\to Y^*(p)$ is continuous for any admissible $p$. Furthermore, for any continuous  perturbation of the fundamental parameter $p$ (keeping $p$ admissible), the $F$-stable nondegenerate point $Y^*(p)=E_{k(p)}$ may change its index $k(p)$ only along the edges of the graph $\Gamma$ drawn in Fig.~\ref{graph}.
\end{theorem}

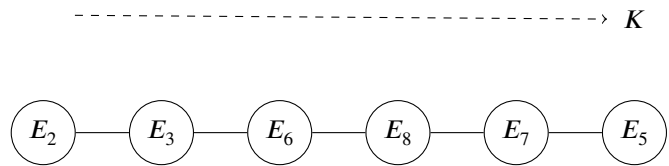
\begin{figure}[ht]
\begin{center}
\begin{tikzpicture}[node distance=7mm]

\node[draw,circle,text opacity=1] (2) at (0,0) {$E_2$};
\node[draw,circle,text opacity=1,right =of 2 ] (3) {$E_3$};
\node[draw,circle,text opacity=1, right =of 3] (6) {$E_6$};
\node[draw,circle,text opacity=1, right  =of 6] (8) {$E_8$};
\node[draw,circle,text opacity=1, right =of 8] (7) {$E_7$};
\node[draw,circle,text opacity=1, right  =of 7] (5) {$E_5$};
\draw[-] (2) -- (3) -- (6) -- (8)-- (7)-- (5); 

\node[circle,above= of 2] (P) {$\phantom{S_2}$};
\node[circle,above= of 5] (Q) {$K$};
\draw[->,dashed] (P) -- (Q);

\end{tikzpicture}
\end{center}
\caption{The transition graph $\Gamma$ of $F$-stable points as a function of the carrying capacity $K$.}
\label{graph}
\end{figure}

\begin{proof}
First note that an $F$-stable point is uniquely determined as the (unique) solution $Y^*(p)$ of  the system
\begin{equation}\label{FEsys}
\begin{split}
Y^*(p)&\ge0\\
F(p,Y^*(p))&\le 0\\
Y_i^*(p)F_i(p,Y^*(p))&=0, \quad \forall i=0,1,2,3,
\end{split}
\end{equation}
where $F_i(p,Y)$ are obtained from \eqref{gen11}. Let $p_k\to p_0$ be a sequence of admissible points converging to an admissible value $p_0$. Then each $Y^*(p_k)$ satisfies \eqref{FEsys} for $p=p_k$. It also follows from \eqref{max} that $Y^*(p_k)$ is a bounded subset of $\R{4}_+$, thus has an accumulation point, say $\lim_{i\to \infty} Y^*(p_{k_i})=\hat Y$ for some subsequence $k_i\to\infty$. Since the left hand sides of \eqref{FEsys} ar continuous functions, $\hat Y$ satisfies \eqref{FEsys}  for $p=p_0$, therefore, by the uniqueness, $\hat Y=Y^*(p_0)$. This also implies that there can exist at most one accumulation point of $\{Y^*(p_k)\}_{k\ge 1}$, therefore $Y^*(p_k$ must converge to $Y^*(p_0)$, the continuity of $p\to Y^*(p)$ follows.

\end{proof}

In particular, it is important to describe the dependence  $K\to Y^*(K)$ when all other parameters
$$
q=(\mu_i,\alpha_j,\eta_k)\in \R{8}_+,\quad 1\le i\le 3, \,\,1\le j\le 3, \,\,1\le k\le 2,
$$
 are fixed and admissible (i.e. \eqref{assum} is satisfied).  A closer inspection of \eqref{Kineq} and \eqref{Ediag} reveals the following monotonicity result.

\begin{theorem}
\label{pro:K}
The  susceptible class $Y_0^*(p)$ is a nondecreasing function of $K$. More precisely, $Y_0^*(p)$ is locally strongly monotonic increasing if $Y^*(p)=E_k$ with $k\in\{2,6,7\}$ ant it is locally constant if $Y^*(p)=E_k$ with $k\in \{3,5,8\}$.
\end{theorem}

According to Theorem~\ref{pro:K}, there can exist only three following transition scenarios  (depending on the choice of $q\in \R{8}_+$). Namely, if $K$ is increasing in $(0,\infty)$ then exactly one of the following alternatives hold:

\begin{itemize}
\item[(i)]
$E_2\to E_3$;
\item[(ii)]
$E_2\to E_3\to E_6\to E_8$;
\item[(iii)]
$E_2\to E_3\to E_6\to E_8\to E_7\to E_5$;
\end{itemize}
The corresponding graphs of the susceptible class $Y_0^*(p)$  are pictured in Figure~\ref{fig3}. Recall that the modified carrying capacity $S_2=K(1-\frac{\mu_0}{b})$ is proportional to $K$.

\begin{figure}[ht]
\begin{center}
\begin{tikzpicture}[scale=0.55]
\draw[->] (-0.2,0) -- (4.2,0) coordinate (x axis) node[above] {$K$};
 \draw[->] (0,-0.2) -- (0,6.5) coordinate (y axis) node[left] {$y$};

\draw[thick,color=black,name path=plot,domain=0:2] plot ({\x},{\x});
\node[below, rotate=0]  at (1,0) {$E_2$};
\draw[thick,color=black,dashed,domain=0:2] plot ({2},{\x});
\draw[thick,color=black,domain=-0.1:0.1] plot ({\x},{2});
\node[left, rotate=0]  at (0,2) {$\sigma_1$};
\draw[thick,color=black,dashed,domain=0:2] plot ({\x},2);

\draw[thick,color=black,domain=2:4.1] plot ({\x},{2});
\node[below, rotate=0]  at (3,0) {$E_3$};

\end{tikzpicture}
\begin{tikzpicture}[scale=0.55]
\draw[->] (-0.2,0) -- (7.2,0) coordinate (x axis) node[above] {$K$};
 \draw[->] (0,-0.2) -- (0,6.5) coordinate (y axis) node[left] {$y$};

\draw[thick,color=black,name path=plot,domain=0:2] plot ({\x},{\x});
\node[below, rotate=0]  at (1,0) {$E_2$};
\draw[thick,color=black,dashed,domain=0:2] plot ({2},{\x});
\draw[thick,color=black,domain=-0.1:0.1] plot ({\x},{2});
\node[left, rotate=0]  at (0,2) {$\sigma_1$};

\draw[thick,color=black,domain=-0.1:0.1] plot ({\x},{3.4});
\node[left, rotate=0]  at (0,3.4) {$\sigma_3$};
\draw[thick,color=black,dashed,domain=0:7] plot ({\x},3.4);

\draw[thick,color=black,domain=2:4] plot ({\x},{2});
\node[below, rotate=0]  at (3,0) {$E_3$};
\draw[thick,color=black,dashed,domain=0:2] plot ({4},{\x});
\draw[thick,color=black,dashed,domain=0:2] plot ({\x},2);

\draw[thick,color=black,domain=4:5.6] plot ({\x},{0.5*\x});
\node[below, rotate=0]  at (4.8,0) {$E_6$};
\draw[thick,color=black,dashed,domain=0:2.8] plot ({5.6},{\x});
\draw[thick,color=black,domain=5.6:7] plot ({\x},{2.8});
\node[below, rotate=0]  at (6.5,0) {$E_8$};

\end{tikzpicture}
\begin{tikzpicture}[scale=0.55]
\draw[->] (-0.2,0) -- (10,0) coordinate (x axis) node[above] {$K$};
 \draw[->] (0,-0.2) -- (0,6.5) coordinate (y axis) node[left] {$y$};

\draw[thick,color=black,name path=plot,domain=0:2] plot ({\x},{\x});
\node[below, rotate=0]  at (1,0) {$E_2$};
\draw[thick,color=black,dashed,domain=0:2] plot ({2},{\x});
\draw[thick,color=black,domain=-0.1:0.1] plot ({\x},{2});
\node[left, rotate=0]  at (0,2) {$\sigma_1$};
\draw[thick,color=black,dashed,domain=0:2] plot ({\x},2);

\draw[thick,color=black,domain=-0.1:0.1] plot ({\x},{3.4});
\node[left, rotate=0]  at (0,3.4) {$\sigma_3$};
\draw[thick,color=black,dashed,domain=0:8.5] plot ({\x},3.4);

\draw[thick,color=black,domain=2:4] plot ({\x},{2});
\node[below, rotate=0]  at (3,0) {$E_3$};
\draw[thick,color=black,dashed,domain=0:2] plot ({4},{\x});

\draw[thick,color=black,domain=4:5.6] plot ({\x},{0.5*\x});
\node[below, rotate=0]  at (4.8,0) {$E_6$};
\draw[thick,color=black,dashed,domain=0:2.8] plot ({5.6},{\x});
\draw[thick,color=black,domain=5.6:7] plot ({\x},{2.8});
\node[below, rotate=0]  at (6.5,0) {$E_8$};
\draw[thick,color=black,dashed,domain=0:2.8] plot ({7},{\x});

\draw[thick,color=black,domain=7:8.5] plot ({\x},{(0.4)*\x});
\draw[thick,color=black,dashed,domain=0:3.4] plot ({8.5},{\x});
\node[below, rotate=0]  at (7.7,0) {$E_7$};

\draw[thick,color=black,domain=8.5:10] plot ({\x},3.4);
\node[below, rotate=0]  at (9.4,0) {$E_5$};
\end{tikzpicture}
\end{center}
\vspace*{-0.5cm}
\caption{Three possible scenarios of the transition dynamics}
\label{fig3}
\end{figure}
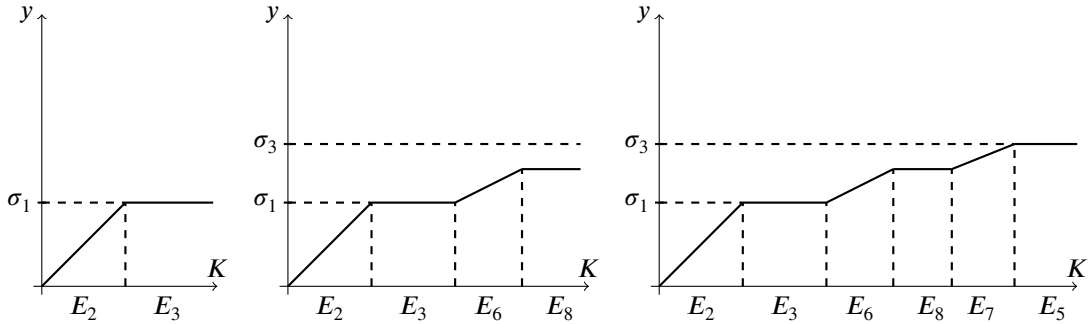

We emphasize  the monotonic non-decreasing dependence of $Y_0^*(p)$ as a function of $K$. It is also interesting to point out that the value of  $Y_0^*(p)$ stabilizes when $K\ge K^*(q)$. In other words, after a certain threshold value $K^*(q)$, the equilibrium point $Y^*(K,q)$ still depends on $K$ except for the susceptible class which becomes constant.

 Let $p$ be a fixed admissible vector. Then it easily follows from \eqref{Ediag} and \eqref{Fdiag} that if $0<S_2<\sigma_1$ then $Y^*(p)=E_2$ is the $F$-stable point. If $S_2=\sigma_1$ then $E_2=E_3$ is the $F$-stable point. Similarly, using \eqref{e6} we also see that if
 $$
 \sigma_1<S_2< \sigma_1':=\sigma_1+(\sigma_3-\sigma_1)\frac{K\alpha_1\alpha_3}{b\eta_1}
 $$
 then $Y^*(p)=E_3$. When $S_2=\sigma_1'$, we have $S_3=S_6$, and it follows from \eqref{e6} that in fact $Y^*(p)=E_3=E_6$. If $S_2$ becomes a bit large than $\sigma_1'$ then $Y^*(p)=E_6$.

Let us consider as an example the behaviour of the equilibrium point   $Y^*(p)=E_5$. Suppose that $Y^*(p)=E_5$ is $F$-stable for some $p$ but it is  a degenerate point. Since $(E_5)_0$ and $(E_5)_3$ are positive, the degeneracy means that, for instance, $(F(E_5))_1=0$. This yields $S_5=S_6$. We claim that in that case $E_5=E_6$. Indeed, one trivially has $(E_5)_0=S_5=S_6=(E_6)_0$ and  $(E_5)_i=0=(E_6)_i$ for $i=1,2$. Also, using  \eqref{e6} we get
$$
(E_5)_3=(S_2-S_5)\frac{b}{K\alpha_3}=
(S_2-S_6)\frac{b}{K\alpha_3}=(S_5-S_3)\frac{\alpha_1}{\eta_1}=
(S_6-S_3)\frac{\alpha_1}{\eta_1}=(E_6)_3.
$$
Thus, $Y^*(p)=E_5=E_6$. In particular, $(F(E_6))_2=(F(E_5))_2\le 0$. If the latter inequality is strong then $Y^*(p)=E_6$, hence $Y^*(p)$ is nondegenerate. If  $(F(E_5))_2= 0$ then $S_5=S_6=S_7$ which  imply by the same argument that $E_5=E_6=E_7$. Then it follows from \eqref{e67} that $\sigma_3\Delta=\delta$. Using Remark~\ref{rem2},  $\delta$ and $\Delta$ are nonzero, therefore $S_8=\frac{\delta}{\Delta}=\sigma_3=S_5$. In particular, this yields from \eqref{e83} that $(\sigma_3-\sigma_1)\Delta=(\sigma_2-\sigma_1)\eta_1\alpha_2$, hence $\Delta>0$. It follows that $E_8=E_5$. But the latter (since  $E_8$ missing the $F$-components) means that $Y^*(p)=E_8$ is nondegenerate.

In general, using the above argument  implies that if $E_k$ is $F$-stable but degenerate then there exists $F$-stable and \textit{nondegenrate} $E_m$ with $m>k$ such that $E_m=E_k$. It interesting to understand the corresponding transition dynamics. To this end, let us write $(i,j)\in E_k$ (resp. $\in F(E_k)$) if $c(S_i-S_j)$ is present in $E_k$ (resp. $F(E_k)$). If $(i,j)\in E_k$ (or $F(E_k)$) implies that either $i=k$ or $j=k$ (this  holds even true for $E_8$ modulus relations \eqref{e83} or \eqref{e84}). A simple examination shows that the following subordination principle holds true.

\begin{proposition}\label{superposition}
$(i,j)\in E_j$ if and only if $(i,j)\in F(E_i)$, and if $(i,j)\in E_j$ (resp. in $F(E_j)$) then $i<j$ (resp. $i>j$).
\end{proposition}
Let us denote by $E=\{(i,j): \text{there exists $E_i$ such that } (i,j) \in E_i\}
$ and let $\Gamma$ denote the undirected graph with nodes $\{k:2\le k\le 8\}$  and edges $E$, see Figure~\ref{graph}. From the biological point of view, the graph $\Gamma$ shows the transition dynamics of a $F$-stable point $Y^*(p)$ depending on continuous perturbations of $p$.


The  trivial equilibrium $E_1=\mathbf{0}$ (i.e. when \textit{no disease or susceptible}), the \textit{disease free equilibrium} point $E_2$, the equilibrium states with the presence of 1st, 2nd strain and the presence of coinfection $E_3,E_4,E_5$ and the equilibrium states  corresponding the coexistence of more than 2 classes $E_6,E_7,E_8$.

\section{Some remarks on the infinite carrying capacity}\label{H.C.C}
To complete the above picture, we announce without proofs some results on the case of very high (infinite) carrying capacity, i.e. $K=\infty$. We return to this situation with detailed discussion somewhere else. The dynamics of the limit case $K=\infty$ is completely different and the global stability is failed in this case, see Remark~\ref{reminf}. Still, we have some nice properties which hold for general parameters.

An easy analysis shows that for $K=\infty$ some equilibrium points from \eqref{Ediag} disappear (`pass to infinity'), so that, generically, only following three equilibrium points exist:
\begin{align*}
E'_3&=(\sigma_1,\frac{b-\mu_0}{\alpha_1},0,0), \\
E'_5&=(\sigma_3,0,0,\frac{b-\mu_0}{\alpha_3}), \\
E'_8&=(\frac{\delta}{\Delta},
\frac{\gamma_2}{\Delta},
\frac{\gamma_1}{\Delta}, (\sigma_2-\sigma_1)\frac{\alpha_1\alpha_2}{\Delta})
\end{align*}
where we assume that  $$\gamma_1=\eta_1(b-\mu_0)-\alpha_1\alpha_3(\sigma_3-\sigma_1)
$$ and
$$\gamma_2=\alpha_2\alpha_3(\sigma_3-\sigma_2)- \eta_2(b-\mu_0)
$$
are nonzero quantities.  The corresponding nontrivial $F$-parts are
\begin{align*}
F(E_3)&=(0,0,(\sigma_1-\sigma_2)\alpha_2,\frac{\gamma_1\alpha_3}{\alpha_1}), \\
F(E_5)&=(0,-\frac{\gamma_1}{\alpha_3},\frac{\gamma_2}{\alpha_3},0).
\end{align*}
In the borderline case $\gamma_1\gamma_2=0$, analysis is somewhat more delicate, here there exist two points which correspond to $E_6$ and $E_7$ for $K<\infty.$

Then the stability diagram is shown in Figure~\ref{fig4}.
\begin{figure}[ht]
\begin{center}
\begin{tikzpicture}[scale=0.45]
\draw[->] (-0.2,0) -- (10,0) coordinate (x axis) node[above] {$\eta_1$};
 \draw[->] (0,-0.2) -- (0,8) coordinate (y axis) node[left] {$\eta_2$};

\draw[thick,dashed,color=black,name path=plot,domain=0:8] plot ({3.5},{\x});
\draw[thick,dashed,color=black,name path=plot,domain=3.5:10] plot ({\x},{3});

\node[right, rotate=0]  at (1.2,6) {$E'_3$};
\node[right, rotate=0]  at (2.3,3.2) {$P$};

\node[right, rotate=0]  at (5.5,1.5) {$E'_8$};

\node[right, rotate=0]  at (5.5,6) {$E'_5$};
\end{tikzpicture}
\end{center}
\vspace*{-0.5cm}
\caption{Three  equilibrium states for $K=\infty$ (the  point $P$ is given by $\gamma_1=\gamma_2=0$)}
\label{fig4}
\end{figure}
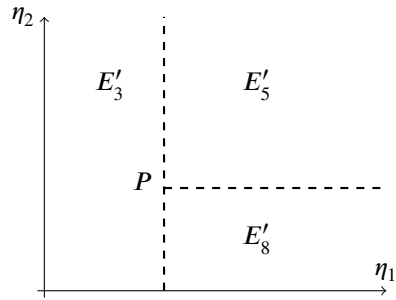

\begin{proposition}
Suppose that $K=\infty$ holds. Then for any  $q\in \R{10}_+$ such that $\gamma_i\ne0$, $i=1,2$, there exists a unique $F$-stable point $Y^*(p)\in \{E'_3,E'_5,E'_8\}$, see the Figure~\ref{fig4}. Moreover, in each of the three cases, the   $\omega$-limit set~ $\omega(Y)$ of a solution to \eqref{model} is one of the following:
	\begin{enumerate}
		\item[(a)]  If $Y^*(p)=E'_3$  then,~
		$$
		\omega(Y)=\{y\in \R{4} : y_2=y_3=0,~y_0-Y_0^*\ln y_0+ y_1-Y_1^*\ln y_1=C_1\},
		$$
	
		\item[(b)] If  $Y^*(p)=E'_5$ then,~
		$$
		\omega(Y)=\{y\in \R{4} : y_1=y_2=0,~y_0-Y_0^*\ln y_0+ y_3-Y_3^*\ln y_1=C_2\},
		$$

		\item[(c)] 	If $Y^*(p)=E'_8$ then,~
		$$
		\omega(Y)=\{y\in \R{4} :~y_0-Y_0^*\ln y_0+ y_1-Y_1^*\ln y_1+y_2-Y_2^*\ln y_2+ y_3-Y_3^*\ln y_3=C_3\},
		$$
	\end{enumerate}
where $C_i$ are some constants.
\end{proposition}

Thus, the $\omega$-limit sets are either one-dimensional curves in the first two cases or a compact hypersurface in $\R{4}_+$ in the latter case.

\section{Discussion}
In this paper we proposed a SIR model with coinfection infection mechanism and observed the effect of density dependence population regulation on disease dynamics. The complete stability analysis of boundary equilibrium points and coexistence equilibrium point revealed that there is always a unique $F$-stable equilibrium point for any admissible set of parameters. The existence of an endemic equilibrium point guarantees the persistence of the disease with a possible future threat of any outbreak in the population.
In the absence of dual infection, exclusion of a strain with an invading strain was also observed which proves the existence of competitive exclusion principle.
Furthermore, we have also shown that addition of a density dependence factor in the susceptible population has played an important role in the disease dynamics. Increase in carrying capacity, increases the number of $F$-stable equilibrium points which makes the dynamics even more complicated. If carrying capacity is significantly high, then the oscillation in different classes are observed and these oscillations dampen to equilibrium point but it approaches the equilibrium point very slowly. We also find that increasing the resources of the population, increased carrying capacity for example by increased wealth, can increase the risk of infection which leads to a destabilization of a healthy population. This becomes especially interesting in the limit case when carrying capacity is very large. Then the healthy population is independent of carrying capacity (increased wealth), since the susceptible population remains constant for very large $K$. Instead the infected population increase as it depends on carrying capacity yet may reach a limit for infinitely large carrying capacities. An increased wealth may therefore only fuel the number of infected in the population which is very different from general expectations. These dynamics resembles the top down regulation in food chains and food webs [22]. In the limit case when $K=\infty$ we have observed the periodic behaviour of solution trajectories in the limit, which becomes even more complex for coexistence equilibrium point. These results indicate that this system has dynamics closely related to the Rosenzweigs [19] famous paradox of enrichment for predator prey models and Sharpe et al in [21] have also found the paradox for a disease model. In the future, we would also like to see if it is possible to analysis this model with more complexity by adding the interaction between two strains and to conduct the global stability analysis for that extended model.




\subsection*{Conflict of interest}

The authors declare no potential conflict of interests.

%

%

\appendix

\section{Two auxiliary lemmas\label{app1}}

The following results are in the spirit of the well-known Landau-Kolmogorov type estimates. On the other hand, our focus is non on an inequality but on  the convergence at infinity. We were  unable to find any explicit formulations like the lemmas below. On the other hand, these results are interesting by their own right and are very useful tools in integral estimates of a rather general ODEs including those of Lotka-Volterra type.
\begin{lemma}\label{lemma1}
	If $f(t) \in L^p ([0,\infty))\cap C^1([0,\infty))$, where $p\geq 1$ and  $f'\in L^\infty([0,\infty))$  then there exist $\lim_{t \rightarrow \infty} f(t) =0.$
\end{lemma}
\begin{proof}
	Let $M=\|f'\|_{L^\infty([0,\infty))}.$
	Arguing by  by contradiction, we can suppose  that there exist $t_k\nearrow\infty$ such that $t_{k+1}-t_k>\frac{\delta}{M}$ and $|f(t_k)|\geq \delta,$ for some fixed $\delta>0.$ 	Since the first derivative is bounded: $|f'(t)|\leq M$ for all $t\ge0$, we have by the mean value theorem for any $t$, $t_k\le t\le t_k+\frac{\delta}{M}$,
\begin{align*}
f(t)&=f(t_k)+f'(\xi)(t-t_k)\\
&\ge \delta-M(t-t_k)\\
&\geq M(t_k-t+\frac{\delta}{M}).
\end{align*}
Note that by virtue of our choice of $t$ the right hand side of the latter inequality is nonnegative. 	Therefore,
$$
\int_{t_k}^{t_k+\frac{\delta}{M}} |f(t)|^p dt\geq\int_{t_k}^{t_k+\frac{\delta}{M}} M^p(t_k +\frac{\delta}{M}-t)^p dt =	M \int_{0}^{\frac{\delta}{M}} s^p ds=
	\frac{M}{p+1}\left(\frac{\delta}{M}\right)^{p+1}.
$$
Since the latter estimate holds uniformly for any $k=1,2,3,\ldots$, this implies $\int_{0}^{\infty} |f|^p dt$ diverges, a contradiction.
\end{proof}

\begin{lemma}\label{lem:Lpbis}
Let $h\in  L^q ([0,\infty))$, where $q\ge1$, and let $h$ have the bounded derivatives $h', h''$. Then
$$
\lim_{t\to\infty}h(t)= \lim_{t\to\infty}h'(t)=0.
$$
\end{lemma}
\begin{proof}
Recall the standard notation:  $x^{[\alpha]}=|x|^{\alpha-1}x$, then ${x^{[\alpha]}}'=\alpha|x|^{\alpha-1}$. We have
	\begin{equation}
	\begin{split}
	\int_{t_0}^{t} |h'|^{2q}\,dt&=
\int_{t_0}^{t}\left( \frac{d}{dt}\bigl(h'^{[2q-1]}h\bigr)- (2q-1)h''h|h'|^{2q-2}\right)dt,\\
&=\left.h'^{[2q-1]}h\right|_{t_0}^{t}-(2q-1)\int_{t_0}^{t} h''h|h'|^{2q-2}dt.
	\end{split}
	\end{equation}
Using the boundedness of $h,h'$ and subsequently Holder's and Young's inequalities gives
\begin{align*}
\int_{0}^{t} |h'|^{2q}\,dt &
\leq C_1+ C_2\biggl(\int_{0}^{t}|h'|^{2q}\,dt\biggr)^\frac{q-1}{q} \biggl(\int_{0}^{t}|h|^q\,dt\biggr)^\frac{1}{q}\\
&\leq C_1+\frac{C_2(q-1)}{q}\epsilon^\frac{1}{q-1}\int_{0}^{t}|h'|^{2q}\,dt +\frac{C_2}{q\epsilon}\int_{0}^{t}|h|^{q}\,dt,
\end{align*}
implying for sufficiently small $\epsilon$ that
$$
\int_{0}^{t} |h'|^{2q}\,dt\le C_3+C_4\int_{0}^{t}|h|^q\,dt.
$$
Since $\|h\|_{L^q([0,\infty))}<\infty$ we obtain 	$\|h'\|_{L^{2q}([0,\infty))}<\infty$, therefore applying Lemma~\ref{lemma1} to $f=h'$ and $p=2q$ we deduce that $ \lim_{t\to \infty}h'(t) = 0$. Furthermore, by the made assumptions, Lemma~\ref{lemma1} is also applicable to $f=h(t)$ and $p=q\ge1$, therefore we have $ \lim_{t\to \infty}h(t) = 0$, the lemma follows.
\end{proof}

Repeating the argument of the previous lemma one easily arrives to

\begin{corollary}\label{cor:Lpbis}
Let $h\in  L^q ([0,\infty))$, where $q\ge1$, and let $h$ have the bounded derivatives $h^{(k)}$ of any order $k\ge1$. Then
$$
h^{(k)}\in  L^{2^kq} ([0,\infty)) \quad \text{and}\quad\lim_{t\to\infty}h^{(k)}(t)=0, \quad\quad\forall k\ge0.
$$
\end{corollary}

\nocite{*}


%

\end{document}